\documentclass[a4paper,11pt]{article}
\usepackage{amsmath,amsthm,amssymb}

\theoremstyle{plain}
\newtheorem{theorem}{Theorem}%[section]

\theoremstyle{definition}
\newtheorem{definition}[theorem]{Definition}

\newcommand{\SL}{\operatorname{SL}}
\newcommand{\actson}{\curvearrowright}

\newcommand{\Z}{\mathbb{Z}}
\newcommand{\al}{\alpha}

\newcommand{\recht}{\rightarrow}
\newcommand{\R}{\mathbb{R}}
\newcommand{\Q}{\mathbb{Q}}
\newcommand{\T}{\mathbb{T}}
\newcommand{\rL}{\mathord{\text{\rm L}}}
\newcommand{\M}{\mathord{\text{\rm M}}}
\newcommand{\Hh}{\hat{H}}
\newcommand{\cH}{\mathcal{H}}
\newcommand{\si}{\sigma}
\newcommand{\cU}{\mathcal{U}}
\newcommand{\C}{\mathbb{C}}
\newcommand{\cR}{\mathcal{R}}
\newcommand{\cL}{\mathcal{L}}
\newcommand{\GL}{\operatorname{GL}}
\newcommand{\om}{\omega}
\newcommand{\Ker}{\operatorname{Ker}}
\newcommand{\lspan}{\operatorname{span}}

\begin{document}

\begin{center}
{\LARGE\bf Rigid actions need not be strongly ergodic}

\bigskip

{\sc by Adrian Ioana$^{(1)}$ and Stefaan Vaes$^{(2)}$
\setcounter{footnote}{1}\footnotetext{Supported by a Clay Research Fellowship.\\ Mathematics Department; University of California at Los Angeles, CA 90095-1555 (United States). \\ E-mail: adiioana@math.ucla.edu}
\setcounter{footnote}{2}\footnotetext{Partially supported by ERC Starting Grant VNALG-200749, Research Programme G.0231.07 of the Research Foundation --
Flanders (FWO) and K.U.Leuven BOF research grant OT/08/032.\\ Department of Mathematics;
    K.U.Leuven; Celestijnenlaan 200B; B--3001 Leuven (Belgium).
    \\ E-mail: stefaan.vaes@wis.kuleuven.be}}
%
%\bigskip
%
%To appear in {\it Proceedings of the American Mathematical Society}
\end{center}

\begin{abstract}\noindent
A probability measure preserving action $\Gamma \actson (X,\mu)$ is called rigid if the inclusion of $\rL^\infty(X)$ into the crossed product $\rL^\infty(X) \rtimes \Gamma$ has the relative property (T) in the sense of Popa. We give examples of rigid, free, probability measure preserving actions that are ergodic but not strongly ergodic. The same examples show that rigid actions may admit non-rigid quotients.
\end{abstract}

\section*{Introduction}

Popa raised the question whether all rigid free ergodic probability measure preserving (pmp) actions are strongly ergodic. We answer this question negatively by providing a counterexample of the form $\Gamma \actson \T^4$ for some well chosen $\Gamma \subset \SL(4,\Z)$. Actually, as we recall below, the notion of rigidity makes sense for all pmp actions. In the same way as property (T) groups can have infinite center, we give examples of rigid actions that are not ergodic. By enlarging the acting group, rigidity is preserved and we obtain ergodic but non strongly ergodic rigid actions.

Popa's definition of relative property (T) for inclusions of finite von Neumann algebras is formulated in the language of bimodules. If $(M,\tau)$ is a von Neumann algebra with a faithful tracial state $\tau$ and $\cH$ is a Hilbert $M$-$M$-bimodule, we say that
\begin{itemize}
\item $\xi \in \cH$ is $M$-central if $x \xi = \xi x$ for all $x \in M$;
\item $\xi \in \cH$ is tracial if $\langle x \xi, \xi \rangle = \tau(x) = \langle \xi x, \xi \rangle$ for all $x \in M$;
\item the sequence of unit vectors $\xi_n \in \cH$ is almost $M$-central if $\|x \xi_n - \xi_n x \| \recht 0$ for all $x \in M$.
\end{itemize}

\begin{definition}[Definition 4.2.1 in \cite{Po01}]
Let $(M,\tau)$ be a von Neumann algebra with a faithful tracial state $\tau$ and $A \subset M$ a von Neumann subalgebra. We say that $A \subset M$ is \emph{rigid} if the following holds. Whenever $\cH$ is a Hilbert $M$-$M$-bimodule and $\xi_n \in \cH$ is a sequence of tracial, almost $M$-central vectors, then $\|\xi_n - P_A(\xi_n)\| \recht 0$, where $P_A$ denotes the orthogonal projection of $\cH$ onto the subspace of $A$-central vectors.
\end{definition}

For a subgroup $H \subset G$ of a countable group $G$, rigidity of $\rL H \subset \rL G$ is equivalent with the relative property (T) of $H \subset G$ in the sense of Kazhdan-Margulis, see \cite[Proposition 5.1]{Po01}. By definition, rigidity of $A \subset M$ is preserved when making $A$ smaller or $M$ larger.

Whenever $\Gamma \actson (X,\mu)$ is a pmp action of a countable group $\Gamma$, we denote by $M=\rL^\infty(X) \rtimes \Gamma$ the crossed product von Neumann algebra. We always equip $M$ with the trace given by
$$\tau(a u_g) = \begin{cases} \int_X a \; d\mu &\;\;\text{if}\;\; g = e \; , \\ 0 &\;\;\text{if}\;\; g \neq e \; .\end{cases}$$

\begin{definition}[\cite{Po01}]\label{def.rigid}
A probability measure preserving action $\Gamma \actson (X,\mu)$ is said to be rigid if the inclusion $\rL^\infty(X) \subset \rL^\infty(X) \rtimes \Gamma$ is rigid.
\end{definition}

By \cite[Theorem 0.3]{Io09} rigidity of an ergodic pmp action $\Gamma \overset{\sigma}{\actson} (X,\mu)$ is equivalent with the following ergodic theoretic property: there is no sequence of probability measures $\nu_n$ on $X \times X$ satisfying the properties
\begin{itemize}
\item $p^i_* \nu_n = \mu$ for all $n$ and $i =1,2$, where $p_i : X \times X \recht X$ denotes the projection on the $i$-th coordinate,
\item $\nu_n(\Delta) = 0$ for all $n$, where $\Delta \subset X \times X$ denotes the diagonal,
\item for all bounded measurable functions $f,g : X \recht \C$, one has
$$\int_{X \times X} f(x) g(y) \; d\nu_n(x,y) \recht \int_X f(x) g(x) \; d\mu(x) \; ,$$
\item $\|(\si_g \times \si_g)_* \nu_n - \nu_n\| \recht 0$ for all $g \in \Gamma$, where $\| \, \cdot \,\|$ denotes the norm of a bounded signed measure on $X \times X$ when viewed as a linear functional on the Banach space of bounded Borel functions $\cL^\infty(X \times X)$ equipped with the supremum norm.
\end{itemize}

One associates to any ergodic pmp action $\Gamma \actson (X,\mu)$ the type II$_1$ orbit equivalence relation $\cR := \cR(\Gamma \actson X)$, giving rise to a II$_1$ factor $\rL \cR$ with Cartan subalgebra $\rL^\infty(X)$. When the action $\Gamma \actson (X,\mu)$ is essentially free, we have $\rL \cR = \rL^\infty(X) \rtimes \Gamma$. In the non-free case this equality does not hold. Moreover, it is an open problem whether in general the rigidity of a non-free ergodic pmp action $\Gamma \actson (X,\mu)$ in the sense of Definition \ref{def.rigid} (i.e.\ the relative property (T) of $\rL^\infty(X) \subset \rL^\infty(X) \rtimes \Gamma$ or the above ergodic theoretic criterion) is equivalent or not with the rigidity of the orbit equivalence relation, i.e.\ the relative property (T) of $\rL^\infty(X) \subset \rL \cR$. We refer to \cite[Remark 4.6]{Io09} for more details.

We finally recall the notion of strong ergodicity. A pmp action $\Gamma \actson (X,\mu)$ is said to be strongly ergodic if there are no non-trivial asymptotically $\Gamma$-invariant subsets in $X$. More precisely, whenever $X_n \subset X$ is a sequence of measurable subsets such that the measure of the symmetric difference $g \cdot X_n \bigtriangleup X_n$ tends to zero for all $g \in \Gamma$, then $\mu(X_n) (1-\mu(X_n)) \recht 0$. Equivalently, there is no bounded sequence of functions $a_n \in \rL^\infty(X)$ such that $\|g \cdot a_n - a_n \|_2 \recht 0$ for all $g \in \Gamma$ and $\int_X a_n \; d\mu = 0$ for all $n$.

In \cite[Theorem 2.1]{JS85} it is shown that an ergodic pmp action $\Gamma \actson (X,\mu)$ is non strongly ergodic if and only if there exists an ergodic pmp action $\Z \actson (Y,\eta)$ and a non-singular quotient map $\pi : X \recht Y$ satisfying $\pi(\Gamma \cdot x) = \Z \cdot \pi(x)$ for a.e.\ $x \in X$. Since actions of amenable groups are never rigid, in an attempt to establish that rigidity implies strong ergodicity, we tried to prove that quotients of rigid actions remain rigid. As we observe in Remark 3 at the end of the article, quotients of rigid actions actually need not be rigid.

A pmp action $\Gamma \actson (X,\mu)$ is strongly ergodic whenever it has spectral gap, i.e.\ if the unitary representation $\Gamma \recht \cU(\rL^2(X) \ominus \C 1)$ does not weakly contain the trivial representation. Since $\SL(2,\Z)$ is non-amenable while all stabilizers for the action of $\SL(2,\Z)$ on $\Z^2 - \{0\}$ are amenable, it follows that the natural action of $\SL(2,\Z)$ on $\R^2/\Z^2 = \T^2$ has spectral gap and in particular is strongly ergodic, see \cite{Sc79}. From this it follows immediately that the action $\SL(n,\Z) \actson \R^n/\Z^n = \T^n$ has spectral gap whenever $n \geq 3$.

If $\Gamma$ has property (T) then all ergodic pmp actions are strongly ergodic. An amenable group $\Gamma$ cannot act strongly ergodically on a standard non-atomic probability space.

\section*{Examples of rigid actions that are ergodic but not strongly ergodic}

Let $k \geq 0$ and $n \geq 2$ be integers. Fix an arbitrary subgroup $\Lambda \subset \SL(k,\Z)$ and define the group $G \subset \SL(k+n+1,\Z)$ given by
$$G = \begin{pmatrix} \Lambda & \M_{k,n}(\Z) & \Z^k \\ 0 & \SL(n,\Z) & \Z^n \\ 0 & 0 & 1 \end{pmatrix} \; .$$
Denote by $I_s$ the $s \times s$ identity matrix and define the following subgroups of $G$.
\begin{align*}
\Gamma = \begin{pmatrix} \Lambda & \M_{k,n}(\Z) & 0 \\ 0 & \SL(n,\Z) & 0 \\ 0 & 0 & 1 \end{pmatrix} \qquad & , \qquad
H = \begin{pmatrix} I_k & 0 & \Z^k \\ 0 & I_n & \Z^n \\ 0 & 0 & 1 \end{pmatrix} \; , \\
H_1 = \begin{pmatrix} I_k & \M_{k,n}(\Z) & \Z^k \\ 0 & I_n & \Z^n \\ 0 & 0 & 1 \end{pmatrix} \qquad & , \qquad
H_2 = \begin{pmatrix} I_k & \M_{k,n}(\Z) & \Z^k \\ 0 & \SL(n,\Z) & \Z^n \\ 0 & 0 & 1 \end{pmatrix} \; .
\end{align*}

Note that $H,H_1,H_2$ are all normal subgroups of $G$ and that we can view $G$ as the semidirect product
$$G = H \rtimes \Gamma = H_1 \rtimes (\Lambda \times \SL(n,\Z)) = H_2 \rtimes \Lambda \; .$$
We identify the group von Neumann algebra $\rL H$ with $\rL^\infty(\Hh)$, where $\Hh \cong \T^{k+n}$ is equipped with the Lebesgue measure. In this way, we get a natural probability measure preserving action $\Gamma \actson \Hh$ which has the property that $\rL^\infty(\Hh) \rtimes \Gamma = \rL G$.

\begin{theorem}\label{ourthm1}
Let $k \geq 0$ and $n \geq 2$ be integers. Fix an arbitrary subgroup $\Lambda \subset \SL(k,\Z)$ and define the groups $G, \Gamma, H, H_1, H_2$ as above.
\begin{enumerate}
\item The pair $H_1 < H_2$ has the relative property (T). In particular, $H < G$ has the relative property (T) and $\Gamma \actson \Hh$ is a rigid action.
\item The action $\Gamma \actson \Hh$ is ergodic if and only if $\{A x \mid A \in \Lambda\}$ is infinite for all $x \in \Z^k - \{0\}$. In particular, if $k \geq 1$ and $\Lambda = \{I_k\}$ is the trivial group, the action $\Gamma \actson \Hh$ is rigid but not ergodic.
\item The action $\Gamma \actson \Hh$ is strongly ergodic if and only if the natural action $\Lambda \actson \widehat{\Z^k} = \T^k$ is strongly ergodic.
\item The action $\Gamma \actson \Hh$ is essentially free.
\end{enumerate}
In particular, if $k,n \geq 2$, if $A \in \SL(k,\Z)$ is a matrix with no eigenvalues being a root of unity and if $\Lambda = A^\Z$, then the action $\Gamma \actson \Hh$ is rigid, free, probability measure preserving and ergodic, but not strongly ergodic.
\end{theorem}

\begin{proof}
1.\ Define the following subgroups of $H_1$.
$$K = \begin{pmatrix} I_k & \M_{k,n}(\Z) & 0 \\ 0 & I_n & 0 \\ 0 & 0 & 1 \end{pmatrix} \qquad , \qquad
L = \begin{pmatrix} I_k & 0 & 0 \\ 0 & I_n & \Z^n \\ 0 & 0 & 1 \end{pmatrix} \; .$$
Embedding $\Z^n$ as the $i$-th row into $\M_{k,n}(\Z)$, we define the subgroup $K_i < K$. Since $n \geq 2$, we know that $\Z^n \subset \Z^n \rtimes \SL(n,\Z)$ has the relative property (T). Therefore, $L < H_2$ and $K_i < H_2$ have the relative property (T). The group $H_1$ is boundedly generated by $L,K_1,\ldots,K_k$. More precisely one checks easily that $K=K_1 \cdots K_k$ and $H_1 = KLKL$. Hence $H_1 < H_2$ has the relative property (T).

2.\ The action $\Gamma \actson \Hh$ is ergodic if and only if $\{g x \mid g \in \Gamma\}$ is infinite for all $x \in \Z^{k+n} - \{0\}$. It is easy to see that the latter is equivalent with $\{A x \mid A \in \Lambda\}$ being infinite for all $x \in \Z^k - \{0\}$.

3.\ Denote the action $\Gamma \actson \rL(\Z^{k+n}) = \rL^\infty(\T^{k+n})$ by $(\si_g)_{g \in \Gamma}$. Denote by $\tau$ the natural tracial state on $\rL(\Z^{k+n})$, corresponding to integration with respect to the Lebesgue measure on $\rL^\infty(\T^{k+n})$. Let $\|\, \cdot \,\|_2$ be the associated $2$-norm. View $\rL(\Z^k) \subset \rL(\Z^{k+n})$ and denote by $E : \rL(\Z^{k+n}) \recht \rL(\Z^k)$ the unique trace preserving conditional expectation. It suffices to prove the following claim: if
$a_i \in \rL(\Z^{k+n})$ is a bounded sequence satisfying $\|a_i - \si_g(a_i)\|_2 \recht 0$ for all $g \in \Gamma$, then $\|a_i - E(a_i)\|_2 \recht 0$. To prove this claim, it suffices to show the stronger statement that $(\si_g)_{g \in \Gamma}$, viewed as a unitary representation on $\ell^2(\Z^{k+n})$, is such that its restriction to $\ell^2(\Z^{k+n}) \ominus \ell^2(\Z^k)$ does not weakly contain the trivial representation. The latter follows because, for $n \geq 2$, the unitary representation $\SL(n,\Z) \recht \cU(\ell^2(\Z^n - \{0\}))$ does not weakly contain the trivial representation.

4.\ Because $\SL(k+n,\Z) \actson \T^{k+n}$ is essentially free and $\Gamma \subset \SL(k+n,\Z)$, also $\Gamma \actson \Hh$ is essentially free.

Concrete example.\ Take $A$ and $\Lambda = A^\Z$ as in the theorem. Since $\Lambda \cong \Z$ is amenable, the action $\Lambda \actson \T^k$ is not strongly ergodic. By 3, also $\Gamma \actson \Hh$ is not strongly ergodic. All eigenvalues of all matrices in $\Lambda - \{I_k\}$ differ from $1$. Hence, $B x \neq x$ whenever $x \in \Z^k - \{0\}$ and $B \in \Lambda - \{I_k\}$. Hence, $\{B x \mid B \in \Lambda\}$ is infinite for all $x \in \Z^k - \{0\}$. By 2, $\Gamma \actson \Hh$ is ergodic.
\end{proof}

\section*{Concluding remarks}

{\bf 1)\ } If $k \geq 0$ and $n \geq 3$, the group $H_2$ defined above has property (T) and its center is isomorphic with $\Z^k$.

{\bf 2)\ } If $n =2$ and if $\Lambda \subset \SL(k,\Z)$ has the Haagerup property, then $G/H_1 \cong \Lambda \times \SL(2,\Z)$ has the Haagerup property, so that $\rL H_1 \subset \rL G$ is an HT inclusion in the sense of \cite[Definition 6.1]{Po01}.

{\bf 3)\ } Let $\Lambda \subset \SL(k,\Z)$ be a subgroup such that $\Lambda x$ is infinite for all $x \in \Z^k \setminus \{0\}$ and $\Lambda \actson \widehat{\Z^k}$ is non-rigid, e.g.\ because $\Lambda$ is amenable. Denote by $\pi : \Gamma \recht \Lambda$ the natural quotient homomorphism. The action $\Gamma \actson \widehat{\Z^k}$ given by $g \cdot \om = \pi(g) \cdot \om$ is a \emph{non-rigid quotient} of the rigid action $\Gamma \actson \widehat{H}$. Since $\Gamma$ acts on $\widehat{\Z^k}$ through the quotient homomorphism $\pi$, the action $\Gamma \actson \widehat{\Z^k}$ is not free. The following simple modification also provides free counterexamples: the diagonal action $\Gamma \actson \widehat{H} \times \widehat{H}$ given by $g \cdot (\om,\om') = (g \cdot \om,g\cdot \om')$ is free, ergodic, rigid and admits the diagonal action $\Gamma \actson \widehat{\Z^k} \times \widehat{H}$ given by $g \cdot (\om,\om') = (\pi(g) \cdot \om,g \cdot \om')$ as a quotient. This quotient action is free and ergodic. To see that it is non-rigid, it suffices to observe that $\Z^k \times H \subset (\Z^k \times H) \rtimes \Gamma$ does not have the relative property (T). The latter holds because
$$\Z^k \times H \subset (\Z^k \times H) \rtimes \Gamma \subset (\Z^k \times \Lambda) \times (H \rtimes \Gamma)$$
and because $\Z^k \times \Lambda$ is amenable.

{\bf 4)\ }
By Theorem \ref{ourthm1}, for $n \geq 3$, the natural action of
$$\Gamma_1 = \begin{pmatrix} 1 & \M_{1,n-1}(\Z) \\ 0 & \SL(n-1,\Z)\end{pmatrix}$$
by automorphisms of $\Z^{n}$ yields the rigid embedding $\rL \Z^{n} \subset \rL \Z^{n} \rtimes \Gamma_1$ and hence a rigid pmp action $\Gamma_1 \actson \T^{n}$. Under the natural isomorphism $\rL \Z^{n} \cong \rL^\infty(\T^{n})$, the automorphism of $\rL \Z^{n}$ given by $A \in \SL(n,\Z)$ gets transformed into the pmp transformation of $\T^{n} = \R^{n} / \Z^{n}$ given by $(A^{-1})^t$. So, the natural action of
$$\Gamma_2 = \begin{pmatrix} 1 & 0 \\ \M_{n-1,1}(\Z) & \SL(n-1,\Z) \end{pmatrix}$$
on $\R^n / \Z^n = \T^{n}$ is rigid. The $\Gamma_2$-invariant functions in $\rL^\infty(\T^{n})$ are precisely the functions that only depend on the first variable. Whenever $\al \in \R - \Q$, multiplication by $(\exp(2\pi i \al x_1),\ldots,\exp(2\pi i \al x_{n}))$ provides an ergodic pmp action of $\Z^{n}$ on $\T^{n}$. This action is in semidirect product position with the action of $\Gamma_2$ that we considered so far. Together they yield an action of
$$\Gamma = \left( \mbox{}\hspace{-6ex}\begin{matrix} 1 & 0 & 0 \\ \Z & 1 & 0 \\ \hspace{7ex}\M_{n-1,2}(\Z)\hspace{-3ex}\mbox{} & & \SL(n-1,\Z) \end{matrix}\right)$$
on $\T^{n}$. This provides a different family of examples of rigid, free, pmp actions that are ergodic but not strongly ergodic.

{\bf 5)\ }
For all $n \geq 4$, Theorem \ref{ourthm1} provides examples of subgroups $\Gamma \subset \SL(n,\Z)$ such that the corresponding action $\Gamma \actson \T^n$ is rigid and ergodic, but not strongly ergodic. We prove that for $n \leq 3$ such examples do not exist. So, let $n \leq 3$ and let $\Gamma \subset \SL(n,\Z)$ be a subgroup such that $\Gamma \actson \R^n/\Z^n = \T^n$ is rigid, but not strongly ergodic. We prove that $\Gamma \actson \T^n$ is not ergodic. Since $\Gamma \actson \T^n$ is rigid, the group $\Gamma$ is non-amenable. In particular, $\Gamma$ has no abelian finite index subgroup. From the strong ergodicity criterion in \cite[Theorem 6.5(ii)]{FS96} it then follows that the linear action of $\Gamma$ on $\R^n$ has a non-trivial invariant subspace. Since $n \leq 3$, this invariant subspace has dimension 1 or co-dimension 1. We find a non-zero vector $v \in \R^n$ such that $v$ is either a common eigenvector for the matrices in $\Gamma$ or a common eigenvector for the matrices in the transpose $\Gamma^t$ of $\Gamma$.

We make use of the following observation that we prove below: if $n \in \{2,3\}$ and if $v \in \R^n - \{0\}$ is a common eigenvector for all the matrices in a non-amenable subgroup $\Gamma \subset \SL(n,\Z)$, then there exists $A \in \SL(n,\Z)$ such that $A^{-1} \Gamma A \subset K$ where
$$K = \begin{pmatrix} \pm 1 & \Z^{n-1} \\ 0 & \GL(n-1,\Z) \end{pmatrix} \; .$$
Using this observation, we may assume that either $\Gamma$ or $\Gamma^t$ is a subgroup of $K$. The projection $\Z^n \recht \Z$ on the first coordinate, extends to a group homomorphism $\Z^n \rtimes K^t \recht \Z \rtimes \Z/2\Z$, implying that $\Z^n \subset \Z^n \rtimes K^t$ does not have the relative property (T). Taking the Fourier transform, it follows that $K \actson \R^n/\Z^n$ is not rigid. So, it is impossible that $\Gamma \subset K$. Hence, we may assume that $\Gamma \subset K^t$. Since $K^t \actson \R^n/\Z^n$ is not ergodic, it follows that $\Gamma \actson \R^n/\Z^n$ is not ergodic.

It remains to prove the observation in the previous paragraph. Let $\om : \Gamma \recht \R^*$ be a group morphism and $v \in \R^n$ a non-zero vector such that $B v = \om(B) v$ for all $B \in \Gamma$. Let $\al_1,\ldots,\al_k \in \R$ be a basis over $\Q$ for the $\Q$-linear span of the coefficients of $v$. Take $v_1,\ldots,v_k \in \Q^n$ such that $v = \al_1 v_1 + \cdots \al_k v_k$. It follows that for all $B \in \Ker \om$ and all $i = 1,\ldots,k$, we have $B v_i = v_i$.
Since $\Ker \om$ is non-amenable and $n \leq 3$, it follows that $\lspan\{v_1,\ldots,v_k\}$ is one-dimensional. We conclude that $v$ is a multiple of a vector in $\Q^n$. Hence, we find a non-zero vector $w \in \Z^n$ such that the greatest common divisor of all coefficients of $w$ equals $1$ and $B w = \om(B) w$ for all $B \in \Gamma$. Denote by $e_1 \in \Q^n$ the first standard basis vector. Take $A \in \SL(n,\Z)$ such that $A e_1 = w$. It follows that
$$A^{-1} \Gamma A \subset \begin{pmatrix} * & \R^{n-1} \\ 0 & \GL(n-1,\R) \end{pmatrix} \; .$$
Since $A^{-1} \Gamma A \subset \SL(n,\Z)$, it follows that actually $A^{-1} \Gamma A \subset K$.

\end{document}